\documentclass[12pt]{article}
\usepackage{amsmath}
\usepackage{amsfonts}
\usepackage{amsthm}
\usepackage{amssymb}
\usepackage{color}

\newtheorem{theorem}{Theorem}[section]

\newtheorem{proposition}[theorem]{Proposition}
\newtheorem{observation}[theorem]{Observation}

\newtheorem{corollary}[theorem]{Corollary}

\theoremstyle{definition}
\newtheorem{definition}[theorem]{Definition}
\theoremstyle{remark}
\newtheorem{remark}[theorem]{Remark}

\numberwithin{equation}{section}
\numberwithin{theorem}{section}
\numberwithin{figure}{section}

\newcommand{\lgc}{{\rm lgca}\hskip0.02cm}

\def\f2{\mathbb{F}_2}

\def\lip{\hskip0.02cm{\rm Lip}\hskip0.01cm}

\newcommand{\bbN}{\mathbb{N}}

\newcommand{\PP}{{\mathcal{P}}}

\newcommand{\tht}{\theta}

\newcommand{\vf}{\varphi}

\newcommand{\bbZ}{\mathbb{Z}}
\newcommand{\gn}{\mathcal{G}_n}
\newcommand\remove[1]{}

\newcommand{\lb}{\label}

\newcommand{\wtw}{if and only if}

\newcommand{\buoo}{without loss of generality}

\newcommand{\DEF}{\buildrel {\mbox{\tiny def}}\over =}

\begin{document}

\title{A characterization of superreflexivity through embeddings of lamplighter
groups}

\author{Mikhail~I.~Ostrovskii and Beata~Randrianantoanina}

\maketitle

\begin{abstract}
We prove that  finite lamplighter groups $\{\mathbb{Z}_2\wr\mathbb{Z}_n\}_{n\ge 2}$
with a standard set of generators
 embed  with uniformly bounded distortions into
any non-superreflexive Banach space, and therefore form a set of test-spaces for
superreflexivity. Our proof is inspired by
the well known identification of  Cayley graphs of  infinite lamplighter groups  with
the horocyclic product of trees. We cover $\mathbb{Z}_2\wr\mathbb{Z}_n$
by three    sets with a structure similar to
a horocyclic product of trees, which    enables
us to construct   well-controlled embeddings.
\end{abstract}

\smallskip



\section{Introduction}


One of the important directions in metric geometry is to find
purely metric characterizations of interesting classes of Banach
spaces. For classes of spaces determined by finite-dimensional
subspaces, this direction is a part of the {\it Ribe program}
which was described by Bourgain \cite{Bou86} who proved the first
metric characterization of superreflexivity. See \cite{Nao12} for
more information on the Ribe program. The goal of this paper is to
prove that lamplighter groups are test spaces for
superreflexivity.

\begin{definition}[\cite{Ost13a}]\label{D:TestSp} Let $\mathcal{P}$ be a class of Banach spaces
and let $T=\{T_\alpha\}_{\alpha\in A}$ be a set of metric spaces.
We say that $T$ is a set of {\it test spaces} for $\mathcal{P}$ if
the following two conditions are equivalent: {\bf (1)}
$X\notin\mathcal{P}$. {\bf (2)} The spaces
$\{T_\alpha\}_{\alpha\in A}$ admit bilipschitz embeddings into $X$
with uniformly boun\-ded distortions.
\end{definition}

Several different sets of test-spaces for superreflexivity
are known: (1) binary trees
\cite{Bou86,Mat99,Bau07,Klo14}, (2) binary diamond and Laakso graphs
\cite{JS09,MN13}, (3) multibranching diamond and Laakso graphs \cite{OR17}. See
\cite{Ost16} for a survey on this matter written in 2014.

In this paper we add one more item to this list: (4) the set  of Cayley graphs of finite  lamplighter
groups. Moreover, we observe that the Cayley graph of an infinite lamplighter
group also is a test space for superreflexivity.

 Some of the characterizations (1)--(3) are independent in the
  sense  that the corresponding families of test
spaces do not admit bilipschitz embeddings into each other with
uniformly bounded distortions. In non-obvious cases this was shown
in \cite{Ost14,LNOO18} for finite binary trees and diamond graphs, and in
\cite{OO17} for diamond and Laakso graphs.


Lamplighter groups are a very interesting class of groups which has been a rich source of important examples in geometric group theory.
In 2008, Naor and Peres  \cite[Section~4]{NP08} proved that the finite lamplighter groups $\mathbb{Z}_2\wr\mathbb{Z}_n$, with
metric defined as a word length with respect to natural sets of
generators, are embeddable into $L_1$ with uniformly bounded
distortions.  The  first goal of this
paper is to strengthen this result and to prove their embeddability into
an arbitrary nonsuperreflexive Banach space with uniformly bounded
distortions. As a consequence we get a new metric characterization of
superreflexivity, see Corollary \ref{C:MetChLamp}.

We consider the following special case of the general wreath
product construction (see, for example, \cite[p.~214]{Har00} for
the general definition).

\begin{definition}[{\cite[p.~129]{DK18}}]\label{D:Wreath} Let $H$ and $L$ be two groups
and $L^H$ be the set of all $L$-valued finitely supported
functions on $H$. Then the {\it  wreath product} $L\wr H$ is
defined as the set $L^H\times H$ equipped with the multiplication
\begin{equation*}
((x_h)_{h\in H}, g)\cdot ((y_h)_{h\in H}, k) := ((x_h\cdot
y_{gh})_{h\in H},gk).
\end{equation*}
\end{definition}

Our main interest will be the wreath product $\mathbb{Z}_2\wr\mathbb{Z}_n$,
for $n\in\bbN$, $n\ge 2$, which we denote by $\mathcal{G}_{n}$,
this group is called the {\it lamplighter group}, see \cite{T17}
for a nice introduction including an explanation for the name.

We identify $\mathbb{Z}_2^{\mathbb{Z}_n}$ with  the  family  of
all subsets of $\mathbb{Z}_n$ by identifying $x = (x_k)_{k\in
\mathbb{Z}_n}$ with $\{j \in \mathbb{Z}_n : x_j = 1\}$, the group
operation on $\mathbb{Z}_2^{\mathbb{Z}_n}$ is  the symmetric
difference. From now on we will abuse notation and treat an
element $x \in \mathbb{Z}_2^{\mathbb{Z}_n}$ as a subset of
$\mathbb{Z}_n$.

Considering an element $(x,k)\in\gn$, we call $x$ the {\it set of
positions where the lamp is on}, and its complement - the {\it set
of positions where the lamp is off}. The number $k\in\bbZ_n$ is called
the {\it location of the lamplighter}.

It is easy to see that the elements $a=(\{0\}, 0)$ and
$t=(\emptyset, 1)$ generate $\gn$. Observe that multiplication by
$a=(\{0\}, 0)$ on the right is the act of changing the lamp {\it
at the current location of the lamplighter} and multiplication by
$t=(\emptyset, 1)$ on the right is the act of the lamplighter
moving one position to the next lamp in the `positive' direction
around the circle.

We consider the metric $\rho$ on $\gn$ defined as the metric of
the left-invariant Cayley graph with respect to the  set of
generators $S=\{t, ta\}$. This means that $x$ is adjacent to $y$
if and only if $x=ys$ or $y=xs$ (i.e $x=ys^{-1}$) for one of the
generators $s\in S$. Observe that the generator $ta$ acts by first
moving one step in the `positive' direction, and then changing the
state of the lamp at the final location of the lamplighter. See
 Section~\ref{S:Step1}  for a more detailed description and an equivalent formula \eqref{E:Rho} for the metric.

The first main result of this paper is

\begin{theorem}\label{T:FinWrNonSR} For any nonsuperreflexive Banach space $X$ the Cayley graphs
of $\mathbb{Z}_2\wr\mathbb{Z}_n$ $(n\ge 2)$ corresponding to the
set $S= \{t, ta\}$ admit embeddings into $X$ with uniformly
bounded distortions.
\end{theorem}

\begin{remark} In \cite{NP08}, when proving embeddability into $L_1$,
Naor and Peres considered $\mathbb{Z}_2\wr\mathbb{Z}_n$ with the
set of generators equal to $\{t,a\}$ instead of  $\{t,ta\}$ as we
do. However, it is easy to see that  the metrics induced by these
two generating sets are bilipschitz equivalent to each other with
a distortion $4$.
\end{remark}

In \cite{NP08}  Naor and Peres  constructed two embeddings into $L_1$, one based on irreducible
 representations   of finite lamplighter groups, and the other motivated by what they refer to as  a `direct geometric reasoning'.
Our embedding technique is quite
different from either of the embeddings in \cite{NP08}.
Our approach is inspired by the description of the Cayley graph
of the infinite lamplighter group $\mathbb{Z}_q\wr\mathbb{Z}$, for $q\in\bbN$, $q\ge 2$, as a  {\it
horocyclic product} of two trees, introduced  by Bartholdi and Woess
\cite{BW05,Woe05}, and by
 the analysis of the
metric structure of horocyclic products of trees by
Stein and Taback in \cite{ST13} who  proved, among other results,
\begin{theorem}\lb{ST}{\rm \cite[Corollary~10]{ST13}}
For any $q\in\bbN$, $q\ge 2$, the Cayley graph
of $\mathbb{Z}_q\wr\mathbb{Z}$ with the set of generators $\{t, ta, \dots, ta^{q-1}\}$ admits an embedding into an $\ell_1$-sum of two $q$-branching trees  with
distortion bounded by $4$.
\end{theorem}

We refer the reader to \cite{T17} and \cite{Woe13} for very nice presentations of  horocyclic products of trees and their applications to lamplighter groups.
We give a more detailed overview of our method  in Section~\ref{outline}.

\begin{remark}\label{R:InfLamp}
It follows from Theorem~\ref{ST}, \cite{Bou86}, and \cite{Ost12}
that for all $q\in\bbN$, $q\ge 2$, the Cayley graph of
$\mathbb{Z}_q\wr\mathbb{Z}$ admits an embedding into any
non-super\-re\-fle\-xive space with distortion independent of $q$,
 as
 was shown in a similar case of hyperbolic groups in
\cite[Section 2]{Ost14} (cf. also  Section~\ref{step3} below).

  The bilipschitz embeddability of the infinite group
$\mathbb{Z}_2\wr\mathbb{Z}$ into any non-super\-re\-fle\-xive
space can also be derived from Theorem~\ref{T:FinWrNonSR} and
results of \cite{Bou86,Ost12}.
\end{remark}

In \cite{LPP96} it is proved that there exists a constant $c>0$ so that  for every $n\in\bbN$ a complete binary tree of depth $cn$ embeds with constant distortion into $\gn$. In Section~\ref{S:Tree-proof} below we show an alternative simple proof of this fact
using our  construction.

As a consequence, and by  Theorem~\ref{T:FinWrNonSR},
Remark~\ref{R:InfLamp}, and \cite{Bou86}, we obtain

\begin{corollary}\label{C:MetChLamp} The sequence of Cayley graphs
for $\{\mathbb{Z}_2\wr\mathbb{Z}_n\}_{n\ge 2}$ with respect to the set
 of generators $S=\{t, ta\}$ is a set of test-spaces for
superreflexivity. The Cayley graph of $\mathbb{Z}_2\wr\mathbb{Z}$
with respect to any finite  set of generators is a test space for
superreflexivity.
\end{corollary}

\begin{remark}
We note that Theorem~\ref{ST} is valid for  $\mathbb{Z}_q\wr\mathbb{Z}$ for all $q\in \bbN$, $q\ge 2$. As we elaborate in Section~\ref{outline}, our proof of
 Theorem~\ref{T:FinWrNonSR} is inspired by the methods of Theorem~\ref{ST}, even though we do not apply its conclusion for our argument.

In our statement and proof of Theorem~\ref{T:FinWrNonSR}, for greater clarity of the presentation, we focused our attention on $\mathbb{Z}_2\wr\mathbb{Z}_n$, but it only requires straightforward adjustments of the proof to obtain the same conclusion as in Theorem~\ref{T:FinWrNonSR} for $\mathbb{Z}_q\wr\mathbb{Z}_n$ with the set of generators $\{t, ta, \dots, ta^{q-1}\}$,   for all $q\in \bbN$, $q\ge 2$. The only difference is that one needs to use  an $\ell_\infty$-sum of   two $q$-branching trees in place of binary trees. Similarly as in \cite{ST13}, the value of $q$ affects the branching of the trees, but not the number of the summands. Since for all $q$, a $q$-branching tree embeds almost isometrically into a binary tree, and since we always use an
$\ell_\infty$-sum of   two  trees,  the uniform bound on distortions does not depend on $q\in\bbN$.

\remove{For greater clarity of the presentation, we elected to present our proof of Theorem~\ref{T:FinWrNonSR} only for $\mathbb{Z}_2\wr\mathbb{Z}_n$, i.e. for $q=2$. However, similarly as in the proof of Theorem~\ref{ST}, the same arguments, with only straightforward  adjustments,   also work for $\mathbb{Z}_q\wr\mathbb{Z}_n$ with the set of generators $\{t, ta, \dots, ta^{q-1}\}$,  for all $q\in \bbN$, $q\ge 2$ (the only difference is that one needs to use $q$-branching trees in place of binary trees).}

Since, for all $q\in \bbN$, $q\ge 2$,  the Cayley graphs of
$\mathbb{Z}_q\wr\mathbb{Z}_n$  and $\mathbb{Z}_q\wr\mathbb{Z}$
with the set of generators $\{t, ta, \dots, ta^{q-1}\}$
  isometrically contain the Cayley graphs of
$\mathbb{Z}_2\wr\mathbb{Z}_n$ and $\mathbb{Z}_2\wr\mathbb{Z}$ with
generators $\{t, ta\}$, respectively, it follows from
Corollary~\ref{C:MetChLamp} and the adjusted version of
Theorem~\ref{T:FinWrNonSR}, that  for any $q\in \bbN$, $q\ge 2$,
the sequence of Cayley graphs of $\mathbb{Z}_q\wr\mathbb{Z}_n$
with the set of generators $\{t, ta, \dots, ta^{q-1}\}$ is a set
of test-spaces for superreflexivity.  Similarly, for any $q\in
\bbN$, $q\ge 2$, the Cayley graph of $\mathbb{Z}_q\wr\mathbb{Z}$,
with respect to any finite generating set, is  a test space for
superreflexivity.
\end{remark}

\section{Proof of Theorem~\ref{T:FinWrNonSR}}\label{S:FinWrNonSR}

\subsection{Outline of the proof}\lb{outline}

To simplify notation, we assume that $n$ is divisible by $6$. It is clear that the same
arguments work in general, but the formulas will be somewhat more
complicated.

We cover the Cayley graph of $\mathbb{Z}_n$ with respect to the
generating set $\{\pm1\}$ by three overlapping paths $P_1,P_2,P_3$ of length
$\frac23\,n$ each in such a way that each pair of points in
$\mathbb{Z}_n$ belongs to at least one of the paths (paths that
exclude three mutually disjoint thirds of the cycle $\mathbb{Z}_n$  work).

We consider the following three subsets of $\gn$ for  $i=1,2,3,$
\begin{equation}\lb{defP}
\mathcal{P}_{i,n}\DEF \{(x,k)\in \gn \ | \ x\subseteq  \mathbb{Z}_n, k\in P_i\}.
\end{equation}

We equip  $\PP_{i,n}$ with the metric inherited from the Cayley
graph of $\gn$ with respect to $S=\{t, ta\}$. We note that  the union of the three sets $\PP_{i,n}$ covers $\gn$.

Our approach is inspired by the well-known description of infinite
lamplighter groups $\bbZ_2\wr \bbZ$ as a horocyclic product of two
infinite trees \cite{BW05,Woe05} and by the bilipschitz embedding
of the word metric on $\bbZ_2\wr \bbZ$ into an $\ell_1$-sum of two
trees  \cite{ST13} (see Theorem~\ref{ST} above). Our first goal is
to cover $\mathbb{Z}_2\wr\mathbb{Z}_n$ by three    sets with a
structure similar to a horocyclic product of trees, which will
enable us to construct   well-controlled embeddings.

The sets $\PP_{i,n}$ are defined so that for all locations $k,l\in
P_i$, the length of the path that  is contained in $P_i$ and
connects $k$ and $l$ is at most twice the length of the shortest
path from $k$ to $l$ in $\bbZ_n$. For this reason, the metric
structure of the sets $\PP_{i,n}$ sufficiently resembles the
metric structure of a subset of $\bbZ_2\wr \bbZ_n$ and we are able
to construct sets $W_n$, which we think of as analogs of
horocyclic products of trees, and which are bilipschitz equivalent
to the sets $\PP_{i,n}$. The sets $W_{n}$ are defined as
specific subsets of the $\ell_\infty$-sum of two trees of depth
$n$.   To identify each element of $\PP_{i,n}$ with an element of
the Cartesian product $T_n\times T_n$ of two trees of depth $n$,
we first mark a vertex $v_0$ in $\bbZ_n$ which is the midpoint of
the complement of the path $P_i$, that is $v_0$ is at the distance
at least $n/6$ from any element $k\in P_i$. Modelling our
description on the identification of  $\bbZ_2\wr \bbZ_n$ with the
horocyclic product of trees (cf. \cite[p.~419]{Woe05}), for every
$(x,k)\in \PP_{i,n}$ we describe the set $x$ as the union of
$x\cap I_{k,+}$ and $x\cap I_{k,-}$, where $I_{k,+}, I_{k,-}$ are
two disjoint arcs in $\bbZ_n$ both with the  endpoints $v_0$ and
$k$. Each set $x\cap I_{k,+}$ and $x\cap I_{k,-}$ is encoded by a
sequence of $0$s and $1$s of  length equal to the number of
vertices in $I_{k,+}$ and  $I_{k,-}$, respectively. This naturally
encodes each element $(x,k)\in \PP_{i,n}$ by two elements of a
binary tree whose levels (in a tree) add up to $n$. We verify in
Section~\ref{S:Step1} that this encoding is metrically faithful on
 $\PP_{i,n}$, that is, we construct bilipschitz embeddings
 $$\vf_{i,n}:\PP_{i,n}\to T_n\oplus_\infty T_n,$$
 with uniformly bounded distortions, see Section~\ref{S:Step1} for details. This completes the first and main step of our proof.

The next step of our proof is   a routine application of the well-known theory of Lipschitz retracts to conclude that, for each $i\in\{1,2,3\}$, the bilipschitz map
 $\vf_{i,n}:\PP_{i,n}\to T_n\oplus_\infty T_n$, constructed in the first step, can be extended to a Lipschitz map $\bar{\vf}_{i,n}$ from the entire $\gn$ into an $\ell_\infty$-sum of two metric trees of depth $n$, see Section~\ref{step2} for details.

In the final step of our proof we define the map $\Phi_n$ from
$\gn$ into  an $\ell_\infty$-sum of six metric trees of depth $n$
by
$$\Phi_n(x,k)\DEF (\bar{\vf}_{1,n}(x,k), \bar{\vf}_{2,n}(x,k),
\bar{\vf}_{3,n}(x,k)).$$

Clearly, the maps $\Phi_n$ are Lipschitz with the same Lipschitz constants as those of
the maps $\bar{\vf}_{i,n}$, for $i=1,2,3$.
Since  the paths $P_1,P_2,P_3$ were chosen   in such a way that for any two elements
$(x,k), (y,l)\in \gn$, there exists $i\in\{1,2,3\}$ so that $(x,k), (y,l)\in \PP_{i,n}$, it follows that the map $\Phi_n$ is co-Lipschitz with same constant as the map
$\vf_{i,n}$ (in our construction all maps $\{\vf_{i,n}\}_{i=1}^3$ have the same
co-Lipschitz  constant), see Section~\ref{step3}.

Hence to finish the proof of Theorem~\ref{T:FinWrNonSR}, it is
enough to verify that for all $n\in\bbN$, the $\ell_\infty$-sum of
six metric trees of depth $n$ embeds  into any non-superreflexive
Banach space $X$ with uniformly bounded distortions. This follows
readily by known techniques and results on bilipschitz
embeddability of trees into any non-superreflexive Banach space,
and on extension of bilipschitz embeddings into Banach spaces from
vertex sets to graphs to the corresponding $1$-dimensional
complexes. In fact it even suffices to prove the existence of
Lipschitz maps that satisfy slightly weaker requirements,   see
Section~\ref{step3} for details.

\subsection{Step 1}\label{S:Step1}

We define the paths $P_1$, $P_2$ and $P_3$,  to be  arcs of
lengths $\frac23n$ in $\bbZ_n$ with endpoints $[\frac16
n,\frac56n]$,  $[\frac12n, \frac16n]$ and $[\frac56n, \frac12n]$,
respectively (recall that we assumed that $n$ is divisible by
$6$).

Let $\tht_n:\bbZ_n\to \bbZ_n$ be the `rotation' of $\bbZ_n$ by an arc of length $n/3$, i.e. for each $k\in \bbZ_n$, $\tht_n(k)= k+\frac n3$, where the addition is in the sense of $\bbZ_n$. The rotation $\tht_n$ induces the  isometry $\overline{\tht_n}$ of $\gn$ onto itself, defined by
\[\overline{\tht_n}(x,k)=(\tht_n(x),\tht_n(k)),\]
where, as usual, $\tht_n(x)$ denotes the image of the set $x$ under the action of
$\tht_n$.

Note, that the paths $P_1$, $P_2$, $P_3$, satisfy
$P_2=\tht_n(P_1)$ and $P_3=\tht_n^2(P_1)$, and
the metric spaces $\{\PP_{i,n}\}_{i=1}^3$,   defined by \eqref{defP},
 are isometric to each other, specifically
 \begin{equation}\lb{isoP}
 \PP_{2,n}=\overline{\tht_n}(\PP_{1,n}),\ \ \ \ \ \
 \PP_{3,n}={\overline{\tht_n}}^{\, 2}(\PP_{1,n}).
 \end{equation}

Let $T_n$ be a binary tree of depth $n$, that is, $T_n$ is the
graph whose vertices are labelled by sequences of 0s and 1s of
lengths $\le n$ with the usual graph distance $d_T$. We consider
the $\ell_\infty$-sum $T_n\oplus_\infty T_n$ defined as the
Cartesian product $T_n\times T_n$ endowed with the metric
\begin{equation*}\label{E:Dinfty} d_\infty((A_1, A_2),(B_1,B_2))\DEF \max\{d_T(A_1,B_1),~
d_T(A_2,B_2)\}.\end{equation*}

We define
$$W_{n}\DEF\left\{(A_1, A_2)\in T_n\oplus_\infty T_n\ : \
|A_1|+|A_2|=n,~~ |A_1|,|A_2|\in
\left[\frac16\,n,\frac56\,n\right]\right\}, $$ where $|A|$ denotes
the length of the sequence $A\in T_n$.

It will be convenient to also use the distance $d_1$ on $W_n$
which is 2-equivalent with $d_\infty$.
\begin{equation*}
d_1((A_1, A_2),(B_1,B_2))\DEF
d_T(A_1,B_1)+d_T(A_2,B_2).
\end{equation*}

We will show that there exist bijections $\vf_{1,n}$ from $\{(\PP_{1,n},d_\infty)\}_n$ onto
$\{(W_n,d_\infty)\}_n$ which have uniformly bounded distortions.

We define a  map $\vf_{1,n}: \PP_{1,n}\to W_n$ as follows: for any
$(x,k)\in  \PP_{1,n}$, the element  $\vf_{1,n}(x,k)\DEF(A_1,A_2)$,
where $A_1=(a_{1,0},\dots,a_{1,k-1})$ is a sequence of length $k$
and $A_2=(a_{2,1},\dots,a_{2,n-k})$ is a sequence of length $n-k$,
defined by
\begin{equation*}
 a_{1,j}=\begin{cases} 1\ \ \ {\text{ if\ }} j\in x,\\
 0\ \ \ {\text{ if\ }} j\notin x,
 \end{cases}\ \ \ \ \
 a_{2,i}=\begin{cases} 1\ \ \ {\text{ if\ }} n-i\in x,\\
 0\ \ \ {\text{ if\ }} n-i\notin x.
 \end{cases}
 \end{equation*}

It is clear that the map $\vf_{1,n}$ is one-to-one and onto. We
will show that $\vf_{1,n}$ is a bilipschitz isomorphism of
$(\PP_{1,n},\rho)$ and $(W_n,d_\infty)$.

For any $A,B\in T_n$, we denote by $\lgc (A,B)$ the length  of the
greatest common ancestor of $A$ and $B$ in $T_n$. In this
notation,
$$d_T(A,B)=(|A|-\lgc (A,B)) + (|B|-\lgc (A,B)),$$
and for $(A_1, A_2),(B_1,B_2)\in W_n$,
\begin{equation}\label{dW}
 \begin{split}
 d_1((A_1, A_2),(B_1,B_2))&=d_T(A_1, B_1)+d_T(A_2,B_2)\\&= 2(n-(\lgc (A_1,B_1)+\lgc (A_2,B_2))).
 \end{split}
 \end{equation}

 In particular, for all  $(A_1, A_2),(B_1,B_2)\in W_n$ we have
 \begin{equation}\label{diamW}
 \begin{split}
 d_1((A_1, A_2),(B_1,B_2))&\le 2n.
 \end{split}
 \end{equation}

To continue we need to estimate the distance $\rho((x,k),(y,l))$,
where $(x,k),(y,l)\in\gn$ and $\rho$ is the distance in the Cayley
graph of $\gn$ with the generating set $\{t,ta\}$. We use the
following observation: to get from $(x,k)$ to $(y,l)$ we need

\begin{itemize}

\item to  traverse at least one of the two paths from $k$ to $l$ on
the $n$-cycle (the graph of $\mathbb{Z}_n$ with respect to the
generating set $\{\pm 1\}$);

\item to visit all positions $j\in \mathbb{Z}_n$ which are not on the selected  path, but which
belong to $x\triangle y$, and to change the state of  all lamps at these positions.

\end{itemize}

Denote by $p_1$ and $p_2$ the lengths of the two distinct paths
from $k$ to $l$ on $\mathbb{Z}_n$, and by $g_1$ and $g_2$ -- the
sizes of the largest ``gaps'' in these paths, that is, the largest
distances between distinct vertices for which there is no element
of $x\triangle y$ in between  (observe that $g_1$ and $g_2$ are at
least $1$ each). With this notation it is easy to see the validity
of the leftmost inequality in
\begin{equation}\begin{split}\label{E:Rho}
\min\{p_1+2(p_2-g_2),~& p_2+2(p_1-g_1)\}\le
\rho((x,k),(y,l))\\&\le\min\{p_1+2(p_2-g_2),~ p_2+2(p_1-g_1)\}+2.
\end{split}
\end{equation}

The rightmost inequality in \eqref{E:Rho} holds  because
discrepancies with the equality in \eqref{E:Rho} can occur only at
one of the endpoints of the `interval' on $\mathbb{Z}_n$
consisting of all vertices that are visited by an optimal tour
from $k$ to $l$ that establishes the distance between $(x,k)$ and
$(y,l)$. One of the cases when the distance exceeds the minimum by
$2$ is the following: the distance from $k$ to $l$ in the positive
direction is significantly smaller than in the negative direction
and $x\triangle y=\{k, l\}$. In this case, if we start at $k$ we
need to   do one step in the negative direction in order to change
the status of the lamp at $k$, and then head back in the positive
direction. Note that the position $l$ is reached by the step $ta$
in order to change the status of the lamp there, so no additional
steps are needed at this endpoint.

For the rest of the proof we fix $(x,k), (y,l)\in \PP_{1,n}$, and
$(A_1, A_2)=\vf_{1,n}(x,k)$, $(B_1, B_2)=\vf_{1,n}(y,l)$ in $W_n$.

\begin{observation}\label{O:CrucObs} The sum $\lgc (A_1,B_1)+\lgc
(A_2,B_2)$ is equal to the number of vertices of the set $E$ constructed as
a union of two, possibly empty, `intervals' in $\bbZ_n$. One of the intervals starts at $0\in
\mathbb{Z}_n$, goes in the `positive' direction and ends at the
first vertex  which belongs to $ x\triangle y \cup
\{k, l\}$,
excluding this vertex, in particular if $0\in x\triangle y\cup  \{k, l\}$, then the interval is empty. The other interval starts at $(n-1)\in
\mathbb{Z}_n$, goes in the `negative' direction and ends at the
first vertex which is in $x\triangle y\cup \{k, l\}$,
this interval includes its end if and only if it  does not belong to $x\triangle y$.
Since the number of vertices in the  interval $E$ is equal to either its length $|E|$, if $E$ is empty, or to $|E|+1$, otherwise, we obtain
\begin{equation}\lb{Eobs}
|E|\le \lgc (A_1,B_1)+\lgc
(A_2,B_2)\le |E|+1.
\end{equation}
\end{observation}

Note that   \eqref{dW} and \eqref{Eobs} immediately imply that
\[d_1(\vf_{1,n}(x,k),\vf_{1,n}(y,l))\ge 2(n-g),\] where $g$ is
the number of  vertices in the largest `interval' $I$ in
$\mathbb{Z}_n$ which is disjoint with $\{k,l\}\cup x\triangle y$. Since it is clear that there exists a `tour' of the lamplighter that travels from $k$ to $l$, passes through every vertex of the difference $x\triangle y$,
 visits each vertex at most twice, and  stays outside $I$ except
possibly one vertex (see the paragraph after \eqref{E:Rho}),   for
all pairs $(x,k), (y,l)\in \PP_{1,n}$ we get
the inequality
\begin{equation*}
\begin{split}
\rho((x,k),(y,l))&\le 2+2(n-g)\le 2+d_1(\vf_{1,n}(x,k),\vf_{1,n}(y,l))\\
&\le
2+2d_\infty(\vf_{1,n}(x,k),\vf_{1,n}(y,l)).
\end{split}
\end{equation*}

To get the inequality in the other direction we observe that
$\rho((x,k),(y,l))$ is at least   $\tau-1$, where $\tau$ is the
number of vertices in the smallest `interval' $J$ on
 $\mathbb{Z}_n$ that contains all elements of  $ \{k,l\}\cup x\triangle y$. Note that
 $J$ is nonempty, and thus $\tau=|J|+1$.

If the interval $J$ does not contain $0$, then $J$ and  the interval $E$ constructed in Observation~\ref{O:CrucObs} have at most one point in common,
and the union $J\cup E$ covers $\bbZ_n$. Thus the
length of $J$ is at least $n-|E|-1$.
In this case, by \eqref{dW} and \eqref{Eobs}, we get
\[\begin{split}\rho((x,k),(y,l))&\ge\tau-1\ge n-|E|-1\\
&\ge
\frac12(d_1(\vf_{1,n}(x,k),\vf_{1,n}(y,l)))-1\\
&\ge
\frac12(d_\infty(\vf_{1,n}(x,k),\vf_{1,n}(y,l)))-1.
\end{split}\]

If the interval $J$ contains $0$, then it contains at least
$\frac13n-1$ vertices, because by the definition of $\PP_{1,n}$ we have
$k,l\in P_1=\left[\frac16\,n,\frac56\,n\right]$. So in this case, by \eqref{diamW}, we get
\[\begin{split}\rho((x,k),(y,l))&\ge\frac13n-2\\&\ge \frac16\max\{d_1((A_1,
A_2),(B_1,B_2)): (A_1, A_2),(B_1,B_2)\in W_n\}-2\\&\ge
\frac1{6}(d_\infty(\vf_{1,n}(x,k),\vf_{1,n}(y,l)))-2.\end{split}\]

This ends the proof  that the bijections
$\vf_{1,n}:(\PP_{1,n},d_\infty)\to (W_n,d_\infty)$ have uniformly bounded distortions.

Since, by \eqref{isoP}, the  spaces $\PP_{2,n}$ and $\PP_{3,n}$
 are isometric to $\PP_{1,n}$, the maps
$ \vf_{2,n}\DEF\vf_{1,n}\circ\overline{\tht_n}$, and $ \vf_{3,n}\DEF\vf_{1,n}\circ\overline{\tht_n}^{\, 2}$ are  bijections from $\PP_{2,n}$ and $\PP_{3,n}$, respectively, onto
$(W_n,d_\infty)$,  and they have the same Lipschitz and co-Lipschitz constants as the maps $\vf_{1,n}$.

\subsection{Step 2}\lb{step2}

We will apply the well-known theory of Lipschitz retracts.

For basic theory of Lipschitz retracts in metric spaces see e.g. \cite[Chapter
1]{BL00} and \cite[Propositions 2.1, 2.2, and the comment at the
top of page 303]{Lan13}. For the convenience
of the reader, we briefly recall the definitions and results that we use.

A metric space $M$ is called {\it   injective}
 if for every metric space
$B$, every $A\subseteq B$, and every Lipschitz function $f:A\to
M$, there exists a Lipschitz extension of $f$, that is, a function
$\bar{f}:B\to M$, so that $\bar{f}|_A=f$ and
$\lip(\bar{f})=\lip(f)$.

A metric space $M$ is called a {\it $\lambda$-absolute
Lipschitz retract} (where $1\le\lambda<\infty$) provided that
whenever $X$ is isometrically contained in a metric space $Y$,
there exists a retraction, $r$, from $Y$ onto $X$, with
$\lip(r)\le\lambda$.

\begin{theorem}[{\cite[Proposition 2.2]{Lan13}}]
A metric space $M$ is   injective \wtw\ it is an absolute $1$-Lipschitz retract.
\end{theorem}


A metric space $M$ is said to have the {\it binary
intersection property} if every collection of mutually
intersecting closed balls in $M$ has a common point.

 A metric space $M$ is said to be {\it
metrically convex} if for every $x_0,x_1\in M$ and for every
$0<t<1$ there is a point $x_t\in M$ such that
$d(x_0,x_t)=td(x_0,x_1)$ and $d(x_1,x_t) = (1-t)d(x_0,x_1)$.

\begin{proposition}[{\cite[Proposition 1.4]{BL00}}] A metric space
$M$ is an absolute $1$-Lipschitz retract if and only it is
metrically convex and has the binary intersection property.
\end{proposition}

A
metric space $M$ is called a {\it metric tree} if it is complete,
metrically convex, and for any pair of vertices there is a unique
continuous curve joining them in $M$. Given a tree $T$ in Graph
Theory sense, one can construct the corresponding metric tree by
attaching between any two adjacent vertices mutually disjoint curves isometric to
the interval $(0,1)$.

\begin{proposition}[{Corollary of \cite[Lemma 2.1]{JLPS02}}] A metric binary
tree of any height  is a $1$-absolute
Lipschitz retract.
\end{proposition}

\begin{observation}
The binary
intersection property and metric convexity are preserved under
$\ell_\infty$-sums.
\end{observation}
\begin{proof}[Proof sketch]
{\it Binary intersection property}: If balls $\{B((x_i,y_i), r_i)\}_i$
in $M_1\oplus_\infty M_2$ are mutually intersecting, then the same
happens for projected balls $\{B((x_i, r_i)\}_i$ in $M_1$ and
$\{B((y_i, r_i)\}_i$ in $M_2$. Thus each of the projected
collections of balls has a nonempty intersection. Let $x$ and $y$
be some points in the intersections. Using the definition of the
$\ell_\infty$-sum we get that $\{B((x_i,y_i), r_i)\}_i$ are
Cartesian products of $\{B((x_i, r_i)\}_i$ and $\{B((y_i,
r_i)\}_i$. Thus $(x,y)$ is in the intersection of  $\{B((x_i,y_i),
r_i)\}_i$.

{\it Metric convexity}: Let $(x_0,y_0)$ and $(x_1,y_1)$ be two points in
the $\ell_\infty$-sum. Let $\{x_t\}$ be a suitable family for
$M_1$ and $\{y_t\}$ be a suitable family for $M_2$. Then
$\{(x_t,y_t)\}$ is a suitable family for the $\ell_\infty$-sum.
\end{proof}

\begin{corollary}
Let $\Pi_n$ be the $\ell_\infty$-sum of two metric
binary trees of depth $n$ each. Then $\Pi_n$ is a
$1$-absolute Lipschitz retract.
\end{corollary}

As an immediate consequence we obtain the objective of Step~2.

\begin{corollary}
 For $i=1,2,3$, let $\vf_{i,n}:\PP_{i,n}\to W_{n}$ be the maps defined in Step~1.
Since $W_{n}\subseteq T_n\oplus_\infty T_n \subseteq \Pi_n$,
   there exist
Lipschitz extensions of $\vf_{i,n}$ to maps
$\bar{\vf}_{i,n}:\gn\to\Pi_n$ which have uniformly bounded Lipschitz
constants.
\end{corollary}

\subsection{The Final Step of the proof of Theorem~\ref{T:FinWrNonSR}}\lb{step3}

By \cite{Bou86} binary trees $\{T_n\}_{n=1}^\infty$  admit
bilipschitz embeddings into any nonsuperreflexive space $X$ with
uniformly bounded distortions (see also \cite{Pis16} and
\cite{Ost16}), and, \buoo, we may assume that these embeddings do
not decrease any distances.

Using Mazur's method of constructing basic sequences (see
\cite[pp.~4-5]{LT77}), as  in \cite[Section 3]{Ost14},
one can construct  embeddings $j_n$, which
do not decrease distances and have uniformly bounded distortions, of an $\ell_\infty$-sum of six copies of $T_n$ into any non-superreflexive space $X$
\[j_n:T_n\oplus_\infty T_n\oplus_\infty T_n\oplus_\infty T_n\oplus_\infty
T_n\oplus_\infty T_n\to X.\]

Using \cite[Lemma 3.3]{Ost13a} or a direct argument, one can extend
 embeddings $j_n$ to bilipschitz embeddings  of
$\Pi_n\oplus_\infty\Pi_n\oplus_\infty\Pi_n$ into $X$ with
uniformly bounded distortions (recall that $\Pi_n$ denotes the $\ell_\infty$-sum of two metric
binary trees of depth $n$ each).

However, for our purpose, it suffices to extend embeddings $j_n$
to Lipschitz maps from $\Pi_n\oplus_\infty\Pi_n\oplus_\infty\Pi_n$ into $X$ with uniformly bounded Lipschitz constants. This
can be done directly: we extend the maps from a
binary tree to the metric edges using `linear interpolation': a point $u_t$
$(0<t<1)$ on the edge joining vertices $u_0$ and $u_1$ in the
metric tree corresponding to $T_n$ with $d(u_0,u_t)=t$ and
$d(u_t,u_1)=1-t$ is mapped onto the corresponding convex
combination of the images of $u_0$ and $u_1$.

We denote the obtained Lipschitz maps by
\[E_n:\Pi_n\oplus_\infty\Pi_n\oplus_\infty\Pi_n\to X.\]
By   construction, the maps $\{E_n\}_n$
are non-contractive on
$T_n\oplus_\infty T_n\oplus_\infty T_n\oplus_\infty
T_n\oplus_\infty T_n\oplus_\infty T_n$, and
their   Lipschitz constants
are bounded by a universal constant $C$.

We are now ready to
define the embeddings $F_n:\gn\to X$. We put
\[F_n(x,k)\DEF E_n(\bar{\vf}_{1,n}(x,k), \bar{\vf}_{2,n}(x,k),
\bar{\vf}_{3,n}(x,k)).\]

It follows   from Step 1 and our construction that the Lipschitz
constants of the maps $F_n$ are uniformly bounded .
To prove that maps $F_n$ are uniformly co-Lipschitz, we observe that for
any $(x,k),(y,l)\in\gn$, there
exists $i\in\{1,2,3\}$ such that both $(x,k)$ and $(y,l)$ are in
$\PP_{i,n}$. By  Step~1 and since the
restriction of $E_n$ to $T_n\oplus_\infty T_n\oplus_\infty
T_n\oplus_\infty T_n\oplus_\infty T_n\oplus_\infty T_n$ is
noncontractive, we get
\[\|F_n(x,k)-F_n(y,l)\|_X\ge d_\infty(\vf_{i,n}(x,k),\vf_{i,n}(y,l))\ge \frac14\rho((x,k),(y,l)),\]
which completes the proof of Theorem~\ref{T:FinWrNonSR}.

\section{Proof of Corollary~\ref{C:MetChLamp}}\lb{S:Tree-proof}

For finite groups.
As  we mentioned in the Introduction,
it is enough to prove that for all $n$, the Cayley graph of $\gn$ contains a subset  that is bilipschitz equivalent with an absolute constant to the
binary tree of depth $n/2$. As mentioned earlier, this fact was proved
in \cite{LPP96} for trees of depth $n/c$, for some $c>0$, but  below we obtain it as a direct consequence of our construction in Section~\ref{S:Step1}.

  Indeed, let $W_n$ be the set defined in Section~\ref{S:Step1}, and  let $U$ be the  subset of   $W_n$ consisting of all pairs $(A,B)\in W_n$ so that
$|A|\le \frac{n}2$, and
$B$ is a sequence consisting of $n-|A|$ zeroes. The distance in $(W_n,d_\infty)$ between
any two elements of the set $U$ is bounded below by the tree distance
of the corresponding $A$s, and is bounded above by twice the the
tree distance of the corresponding $A$s. Thus $U$ is 2-equivalent   to the
binary tree of depth $n/2$. Since, by Step~1, $\PP_{1,n}\subseteq \gn$ is bilipschitz equivalent with $W_n$, and by \cite{Bou86}, the proof for finite groups is complete.

For the infinite group, by Remark~\ref{R:InfLamp}, we only need to show
 that the bilipschitz embeddability of
$\mathbb{Z}_2\wr\mathbb{Z}$ into a Banach space $X$ implies that
$X$ is nonsuperreflexive. This follows similarly as in the case of
finite groups. By \cite{Woe05}, the Cayley  graph of
$\mathbb{Z}_2\wr\mathbb{Z}$ with the generating set  $S=\{t, ta\}$
coincides with the horocyclic product of two infinite binary
trees, i.e. infinite trees whose every vertex has degree 3, and
the identification is obtained by mapping every element $(x,k)\in
\mathbb{Z}_2\wr\mathbb{Z}$ to an element $(A_1,A_2,k)$ of two
infinite sequences of 0s and 1s with an (arbitrary) finite number
of nonzero terms, cf. also \cite{T17}. By Theorem~\ref{ST}, the
metric on $\mathbb{Z}_2\wr\mathbb{Z}$ is bilipschitz equivalent
with the metric inherited from the $\ell_1$-sum of two tree
metrics. Thus, as in the finite case, we see that
$\mathbb{Z}_2\wr\mathbb{Z}$ contains a bilipschitz copy of   an
infinite  rooted binary tree by taking the set of all elements of
the form $(A_1,A_2,k)$, where $k=0,1,2,\dots$,  $A_1$ is any
sequence of 0s and 1s so that all terms with indices larger than
$k$ are equal to 0, and all terms in  the sequence $A_2$ are
equal to 0.

The fact that bilipschitz embeddability of a binary tree into $X$
implies nonsuperreflexivity of $X$ follows from \cite{Bou86}, see
also \cite{Bau07}.

\thanks{ \textbf{Acknowledgements:}
We would like to thank Florent Baudier for suggesting the problem
on lamplighter groups to us. The first named author was supported
by the National Science Foundation under Grant Number
DMS--1700176. Both authors thank the Fields Institute (Toronto)
for partial funding to attend the Workshop on Large Scale Geometry
and Applications, where we started our work on lamplighter
groups.}


\begin{small}

\end{small}

\textsc{Department of Mathematics and Computer Science, St. John's
University, 8000 Utopia Parkway, Queens, NY 11439, USA} \par
  \textit{E-mail address}: \texttt{ostrovsm@stjohns.edu} \par
\smallskip

\textsc{Department of Mathematics, Miami University, Oxford, OH
45056, USA} \par
  \textit{E-mail address}: \texttt{randrib@miamioh.edu} \par

\end{document}